\documentclass[11pt]{article}
\usepackage{amsmath,amsfonts,amsthm,amssymb, mathtools,makecell,setspace}
\usepackage{mathrsfs, graphicx,color,latexsym, tikz, calc, float,hyperref
}
\usepackage{enumerate}

\usepackage{float}

\usetikzlibrary{shadows}
\usetikzlibrary{patterns,arrows,decorations.pathreplacing}
\def\l{\langle} \def\r{\rangle}
\def\b{\beta}
\def\a{\alpha}

\newcommand\Aut{\mathrm{Aut}}
\newcommand\Cay{\mathrm{Cay}}
\newcommand\D{\mathrm{D}}
\newcommand\V{\varepsilon}

\newtheorem{theorem}{Theorem}[section]

\newtheorem{lemma}{Lemma}[section]

\allowdisplaybreaks

\title{\bf \Large  The number of rooted spanning forests of bicirculant graphs \footnote{  The research of Feng was supported by  NSFC (Nos. 12271527, 12071484). T. Wu was supported by NSFC (No. 12261071) and NSF of Qinghai Province (No. 2020-ZJ-920). E-mail addresses: yjlisy@163.com (J. Yang), fenglh@163.com (L. Feng, corresponding author),  lrr999a@163.com (R. Lu), mathtzwu@163.com (T. Wu).}}
\author{
{\small  Jing Yang$^a$, \ \ Lihua Feng$^b$$\dagger$, \ \ Rongrong Lu$^b$,\ \ Tingzeng Wu$^{c,d}$}\\[2mm]
\small  $^a$School of Statistics and Mathematics, Yunnan University of Finance and Economics \\  
    \small   Kunming, Yunnan, 650000, China.\\
\small $^b$School of Mathematics and Statistics, HNP-LAMA, Central South University\\
 \small Changsha, Hunan, 410083, China\\
\small   $^c$School of Mathematics and Statistics, Qinghai Nationalities University\\
\small  Xining, Qinghai, 810007,  China\\
\small $^d$Qinghai Institute of Applied Mathematics, Xining, Qinghai, 810007,  China 
 }

\begin{document}
\maketitle
\begin{abstract}
 A bi-Cayley graph over the cyclic group $(\mathbb{Z}_n, +)$ is called a bicirculant graph. Let
 $\Gamma=BC(\mathbb{Z}_n; R,T,S)$ be a bicirculant graph with $R=-R\subseteq \mathbb{Z}_n\setminus \{0\}$ and $T={-}T\subseteq \mathbb{Z}_n\setminus \{0\}$ and $S\subseteq \mathbb{Z}_n$. In this paper, using Chebyshev polynomials, we obtain a closed formula
  for the number of rooted spanning forests of $\Gamma$. Moreover, we investigate some arithmetic properties of the number of rooted spanning forests of $\Gamma$, and find its asymptotic behaviour as   $n$ tends infinity.\\


\end{abstract}

\section{Introduction}

All graphs considered in this paper are undirected, simple and connected.
Let $\Gamma=(V(\Gamma), E(\Gamma))$ be a graph with vertex set $V(\Gamma)=\{1,2, \ldots, n\}$ and edge set $E(\Gamma)$. The \textbf{adjacency matrix} of $\Gamma$, denoted by $A(\Gamma)=\left[a_{i j}\right]_{V(\Gamma)\times V(\Gamma)}$, where
$$
a_{i j}= \begin{cases}1, & \text { if } i \text { and } j \text { are adjacent in } \Gamma, \\ 0, & \text { otherwise. }\end{cases}
$$
The \textbf{Laplacian matrix} $L(\Gamma$) of $\Gamma$ is $L(\Gamma)=D(\Gamma)-A(\Gamma)$, where $D(\Gamma)={\rm diag} (d_1, d_2,\dots, d_n)$ is the diagonal matrix of vertex degrees. The eigenvalues   of $A(\Gamma)$ and $L(\Gamma)$ are called the
\textbf{ eigenvalues}  and \textbf{Laplacian eigenvalues} of $\Gamma$, respectively.

For a finite group $G$ with identity $e$, the Cayley digraph $\Cay(G, S)$ over $G$ with respect to $S\subseteq G\setminus\{e\}$ is a digraph with vertex set $V(\Cay(G, S))=G$  and edge
set $E(\Cay(G, S))=\{(x,y)\mid x,y\in G, yx^{-1}\in S\}$.  A Cayley digraph $\Cay(G, S)$
satisfying $S = S^{-1}$ is called a \textbf{Cayley graph}. Equivalently, a graph $\Gamma$ is a Cayley
graph if and only if its automorphism group $\Aut(\Gamma)$ contains a regular subgroup.
$\Cay(G, S)$ is connected if and only if $S$ generates $G$. It is
well known that $\Cay(G, S)$
is vertex transitive and $G$ acts regularly on $\Cay(G, S)$ as an automorphism group with
one orbit. 
A graph is said to be a \textbf{bi-Cayley graph}(some authors also use the term \textbf{semi-Cayley graph} \cite{GX,GX2,sem,cay})  over a group $G$ if it admits $G$ as
a semiregular automorphism group with two orbits of equal size. Every bi-Cayley graph admits
the following concrete realization (see \cite[Lemma 2.1]{cay}). Let $R$, $T$ and $S$ be subsets of a group $G$
such that $R = R^{-1}$, $T = T^{-1}$ and $R\cup T$ does not contain the identity element of $G$. Define the
graph $BC(G; R, T, S)$ to have vertex set the union of the right part $G_0= \{g_0 \mid g\in G\}$ and
the left part $G_1 = \{g_1 \mid g\in G\}$, and edge set the union of the right edges $\{{h_0, g_0} \mid gh^{-1}\in R\}$,
the left edges $\{{h_1, g_1} \mid gh^{-1}\in T\}$ and the spokes $\{{h_0, g_1} \mid gh^{-1}\in S\}$. 
To be brief, we shall say that a Cayley (resp. bi-Cayley) graph on
a cyclic group a \textbf{circulant graph} (resp. \textbf{bicirculant graph}).

Let
$$
P(x)=a_0+a_1x+\cdots+a_sx^s=a_s \prod_{k=1}^s\left(x-\alpha_k\right)
$$
 be a nonconstant polynomial with complex coefficients.
The \textbf{Mahler measure} \cite{mah} of $P(x)$ is defined to be
\begin{equation}
M(P):=\exp \left(\int_0^1 \log \left|P\left(e^{2 \pi\mathbf{i} t}\right)\right| d t\right).
\end{equation}
Actually, an alternative form of $M(P)$ has appeared in \cite{leh}.
That is,
\begin{equation}
M(P)=\left|a_s\right| \prod_{\left|\alpha_i\right|>1}\left|\alpha_i\right|.
\end{equation}
The concept of Mahler measure can be naturally extended to the class of \textbf{Laurent polynomials}
$$P(x)=a_0 x^t+a_1 x^{t+1}+\cdots+a_{s-1} x^{t+s-1}+a_s x^{t+s}=a_s x^t \prod_{i=1}^s\left(x-\alpha_i\right),$$
 where $a_s \neq 0$, and $t$ is an arbitrary but not necessarily positive integer.

For a graph $\Gamma$ with $n$ vertices, a \textbf{spanning forest} $\Gamma_1$ of $\Gamma$ is an acyclic subgraph that contains all vertices
 of $\Gamma$.  A spanning forest $\Gamma_1$ is
 called \textbf{rooted} if any tree in $\Gamma_1$ has a root, that is a labeled vertex.
 Connected spanning forest is a \textbf{spanning tree}.
The number of spanning trees (also known as the complexity of graphs) and the number of rooted spanning forests in a connected graph $\Gamma$ are denoted by $\tau_{\Gamma}(n)$ and $f_{\Gamma}(n)$, respectively.
For the invariant $\tau_{\Gamma}(n)$, the famous   Matrix-Tree Theorem states that the number of spanning trees
in a graph can be expressed as the product of its non-zero Laplacian eigenvalues divided
by the number of vertices. For the invariant $f_{\Gamma}(n)$,  the classical result \cite{AK} shows that the rooted spanning forests in the graph can be found with the use of determinant of the matrix $\text{det}(I_n+L(\Gamma))$, where $I_n$ is the identity matrix of size $n$.
 This leads to two problems: For the invariant $\tau_{\Gamma}(n)$, how to find the product of non-zero eigenvalues of the matrix $L(\Gamma)$?  For the invariant $f_{\Gamma}(n)$, how to find the product of eigenvalues of the matrix $I_n +L(\Gamma)$? If the number of vertices of a graph is small, it is easy to handle. However, when $n$ tends infinity, calculating these products directly becomes a significantly complex problem.

For the complexity of graphs, many graph classes have been studied,
including complete multipartite graphs \cite{AlmostM1,AlmostM2}, almost complete graphs \cite{M1}, dihedrant graphs \cite{hua}, wheels \cite{wheel}, fans \cite{fan}, prisms \cite{prism}, ladders \cite{ladder}, M\"obius ladders \cite{ladder1}, lattices \cite{latt}, anti-prisms \cite{ati}, complete prisms \cite{ati1}, Sierpinski gaskets \cite{sie1,sie2}, and grids \cite{grid}.
Some researchers had studied this invariant of graphs from various perspectives.
In 1986, Boesch and Prodinger \cite{wheel} were the first to propose using Chebyshev polynomials to analyze the complexity of graphs, and this idea was subsequently put into practice. By employing  their approach, the complexity of circulant graphs have been investigated in \cite{cir2,cir3,cir4}. In 2012, Guttmann et al. \cite{AJ} investigated  the asymptotic behavior of the complexity for some families of graphs
from the point of view of the Malher measure.

In addition to calculating the number of spanning trees of a given graph, one might also be interested in determining the number of rooted spanning forests of the graph, see \cite{PE,OO,yc}.

Recently, Mednykh et al. had achieved a series of significant accomplishments for $\tau_{\Gamma}(n)$ and $f_{\Gamma}(n)$
\cite{MMA3, YA,OO3,MMA1,MMA2,III,OO1,med},
particularly in circulant graphs \cite{MMA1,MMA2}. They developed a new method to produce explicit formulas
as well as the asymptotic formulas for $\tau_{\Gamma}(n)$  and $f_{\Gamma}(n)$ through the Mahler measure of the associated Laurent polynomial.

Inspired by the work of Mednykh et al. \cite{MMA2}, in this paper, we obtain a closed formula for the number of rooted spanning
forests of bicirculant graphs, denoted by $f_\Gamma(2n)$, in light of Chebyshev polynomials, and investigate its arithmetic properties and asymptotic
behaviour.

The paper is organized as follows. In Section 2, some basic definitions and technical results are given.
In Section 3, we obtain  a closed formula for rooted spanning forests of bicirculant graphs.
 They will be given in light of Chebyshev polynomials.
 In Section 4, we provide some arithmetic properties of the number of rooted spanning forests.
In Section 5, we use explicit formulas for $f_{\Gamma}(2n)$ in order to produce its asymptotic in light of Mahler measure of the associated
polynomials. In Section 6, we illustrate our results by a series of examples.
\section{Preliminaries}
In this section, we list some   results that are helpful in the sequel.
Let $\mathbb{F}$ be a number field, and $\mathbb{F}^{m\times n}$ be the set of $m\times n$ matrices over $\mathbb{F}$.

 \begin{lemma}\label{mat}\cite{GX}
 Let $A, B, C, D\in\mathbb{F}^{n\times n}$ with $AC=CA$. Then
 $$ {\rm det}
 \begin{pmatrix}
 A&B\\
 C&D
 \end{pmatrix}
={\rm det}(AD-CB).
$$
\end{lemma}

A \textbf{circulant matrix} matrix $Y$ is of the shape
$$
Y=\operatorname{circ}(y_1, y_2, \ldots, y_n)=
 \begin{pmatrix}
y_1 & y_2 & y_3 & \cdots & y_n \\
y_n & y_1 & y_2 & \cdots & y_{n-1} \\
y_{n-1} & y_n & y_1 & \cdots & y_{n-2} \\
\vdots & \vdots & \vdots & \ddots & \vdots \\
y_2 & y_3 & y_4 & \cdots & y_1
 \end{pmatrix}.
$$
 Recall that (see \cite{PJ}) the eigenvalues of
$Y$ are
 $\lambda_j=g(\V_n^j)$, $j=0,1,2,\ldots,n-1$,
where $\V_n=\exp(2\pi \textbf{i}/n)$  and $g(x)=y_1+y_2x+\cdots+y_nx^{n-1}$.
Let $\Lambda=\Cay(\mathbb{Z}_n,\{\pm s_1,\pm s_2,\ldots,\pm s_\ell\})$ be a circulant graph.
Note that the vertex $i$ is adjacent to the vertices $i\pm s_1,i\pm s_2,\ldots, i\pm s_\ell~(\bmod~n)$.
Then the adjacency matrix of $\Lambda$ is $\sum_{j=1}^{\ell}\left(Q^{s_j}+Q^{-s_j}\right)$,
where $Q=\operatorname{circ}(\underbrace{0,1,0,\ldots,0}_{n})$.
Let $\Gamma=BC(\mathbb{Z}_n;R,T,S)$ be a bicirculant graph. By \cite[Lemma 3.1]{GX},
the adjacency matrix of $\Gamma$ is
$D=\begin{pmatrix}
A & C \\
C^{\top} & B
\end{pmatrix}$,
where $A, B, C$ are the adjacency matrices of $\Cay(\mathbb{Z}_n,R), \Cay(\mathbb{Z}_n,T), \Cay(\mathbb{Z}_n,S)$, respectively,
and $C^{\top}$ means the transpose of $C$. Based on the above facts, we can obtain the following result directly.

\begin{lemma}\label{adj}
Let $R=-R=\{\ell_1, \ell_2\ldots, \ell_r\} \subseteq \mathbb{Z}_n\setminus \{0\}$,
$T=-T=\{m_1, m_2,\ldots, m_t\}\subseteq \mathbb{Z}_n \setminus \{0\}$ and
$S=\{u_1,u_2\ldots, u_s\}\subseteq \mathbb{Z}_n$.
Then the adjacency matrix of $\Gamma=BC(\mathbb{Z}_n; R,S,T)$ is
$$
A(\Gamma)=\begin{pmatrix}
 \sum_{j=1}^{r}Q^{\ell_j} & \sum_{j=1}^{s}Q^{u_j}\\
 \sum_{j=1}^{s}Q^{-u_j} & \sum_{j=1}^{t}Q^{m_j}
 \end{pmatrix},
$$
where $Q=\operatorname{circ}(\underbrace{0,1,0,\ldots,0}_{n})$.
\end{lemma}

Let $T_n(w) = \cos(n\arccos w)$ be the Chebyshev polynomial of the first kind.
The following lemma provides basic properties of $T_n(w)$.

\begin{lemma}\label{che}\cite{med}
 Let $w=\frac{1}{2}\left(z+\frac{1}{z}\right)$ and $T_n(w)$ be the Chebyshev polynomial of the first kind. Then
 $$T_n(w)=\frac{1}{2}\left(z^n+\frac{1}{z^n}\right),$$
 where $z\in \mathbb{C}\setminus \{0\}$ and $n$ is a positive integer.
\end{lemma}

Let $P(z)$ be a polynomial of degree $k$ with integer coefficients and $P(z)=P(z^{-1})$. Then $P(z)$ have the following form
$$
P(z)=\eta_0+\sum_{j=1}^k\eta_j(z^j+z^{-j}),
$$
where $\eta_0,\eta_1,\ldots,\eta_k$ are integers.
Let $w=\frac{1}{2}\left(z+z^{-1}\right)$ and $U(w)=\eta_0+\sum_{j=1}^k2\eta_jT_j(w)$. By Lemma \ref{che}, we have $P(z)=U(w)$.
The polynomial $U(w)$ is called the \textbf{Chebyshev transform} of $P(z)$.

\begin{lemma}\label{ccc}
Let $P(z)=\eta_0+\sum_{j=1}^k\eta_j(z^j+z^{-j})$ be a polynomial of degree $k$ with integer coefficients. Then we have
$$
\begin{aligned}
\prod_{j=0}^{n-1}P(\V_n^j)
&=(-1)^{nk}\eta_k^n\prod_{\ell=1}^{k}(2T_n(w_\ell)-2),\\
\end{aligned}
$$
where $\V_n=\exp(2\pi \mathbf{i}/n)$ and
$w_1, w_2,\ldots, w_k$ are all the roots of Chebyshev transform of $P(z)=0$.
\end{lemma}
\begin{proof}
Let $P_1(z)=\frac{z^k}{\eta_k}P(z)$. Then $P_1(z)$ is a monic polynomial of degree $2k$ with the same roots as $P(z)$. Hence we have
 $$\prod_{j=0}^{n-1}P_1(\varepsilon_n^j)=\frac{\V_n^{\frac{(n-1)nk}2}}{{\eta_k^{n}}}\prod_{j=0}^{n-1}P(\V_n^j)=
 \frac{(-1)^{(n+1)k}}{\eta_k^{n}}\prod_{j=0}^{n-1}P(\V_n^j).$$
Suppose that the roots of $P_1(z)=0$ are $z_1,z_1^{-1},z_2,z_2^{-1}\ldots,z_k,z_k^{-1}$.
 Then we have
$$
P_1(z)=\prod_{\ell=1}^{k}(z-z_\ell)(z-z_\ell^{-1}).
$$
Let $w=\frac{1}{2}(z+z^{-1})$ and $w_\ell=\frac{1}{2}(z_\ell+z_\ell^{-1})$. Then $w_1, w_2,\ldots, w_k$ are all the roots of $U(w)=\eta_0+\sum_{j=1}^k2\eta_j T_j(w)$. Therefore,
$$
\begin{aligned}
\prod_{j=0}^{n-1}P(\V_n^j)
&=(-1)^{(n+1)k}\eta_k^n\prod_{j=0}^{n-1}P_1(\V_n^j)\\
&=(-1)^{(n+1)k}\eta_k^n\prod_{j=0}^{n-1}\prod_{\ell=1}^k(\V_n^j-z_\ell)(\V_n^j-z_\ell^{-1})\\
&=(-1)^{(n+1)k}\eta_k^n\prod_{\ell=1}^k\prod_{j=0}^{n-1}(z_\ell-\V_n^j)(z_\ell^{-1}-\V_n^j)\\
\end{aligned}
$$
$$
\begin{aligned}
&=(-1)^{(n+1)k}\eta_k^n\prod_{\ell=1}^k(z_\ell^n-1)(z_\ell^{-n}-1)\\
&=(-1)^{(n+1)k}\eta_k^n\prod_{\ell=1}^k(2-z_\ell^n-z_\ell^{-n})\\
&=(-1)^{nk}\eta_k^n\prod_{\ell=1}^k(2T_n(w_\ell)-2).
\end{aligned}
$$
This completes the proof.
\end{proof}

Now we define
\begin{itemize}
\item[] $R_1=\{ \pm \a_1, \pm \a_2,\ldots, \pm \a_r\}$, ~~$T_1=\{\pm \beta_1, \pm \beta_2,\ldots, \pm \beta_t\}$,~~
$S_1=\{\gamma_1, \gamma_2,\ldots, \gamma_s\}$,
\item[] $R_2=\{ \pm \a_1, \pm \a_2,\ldots, \pm \a_r, \frac{n}{2}\}$,~~ $T_2=\{\pm \beta_1, \pm \beta_2,\ldots, \pm \beta_t\}$,~~
$S_2=\{\gamma_1, \gamma_2,\ldots, \gamma_s\}$,
\item[] $R_3=\{ \pm \a_1, \pm \a_2,\ldots, \pm \a_r\}$, ~~$T_3=\{\pm \beta_1, \pm \beta_2,\ldots, \pm \beta_t,\frac{n}{2}\}$,~~
$S_3=\{\gamma_1, \gamma_2,\ldots, \gamma_s\}$,
\item[] $R_4=\{ \pm \a_1, \pm \a_2,\ldots, \pm \a_r, \frac{n}{2}\}$, ~~$T_4=\{\pm \beta_1, \pm \beta_2,\ldots, \pm \beta_t,\frac{n}{2}\}$,~~
$S_4=\{\gamma_1, \gamma_2,\ldots, \gamma_s\}$,
\end{itemize}
where $1\le \a_1<\cdots <\a_r<\frac{n}{2}$, $1\le \b_1<\cdots<\b_t<\frac{n}{2}$, $0\le \gamma_1<\cdots<\gamma_s\le n-1$.

Throughout the rest of this paper, we always assume that
$$
\Gamma_j=BC(\mathbb{Z}_n;R_j,T_j,S_j),
$$
where $j=1,2,3,4$. Obviously, any bicirculant graph must be one of $\Gamma_1,\Gamma_2, \Gamma_3, \Gamma_4$,
 and $n$ is even for $\Gamma_2, \Gamma_3, \Gamma_4$.

Let
$$
 \begin{aligned}
 \mathcal{A}(z)&=2r+s+1-\sum_{j=1}^r\left(z^{\a_j}+z^{-\a_j}\right),\\
\mathcal{B}(z)&=2t+s+1-\sum_{j=1}^t\left(z^{\b_j}+z^{-\b_j}\right),\\
\mathcal{C}(z)&=-\sum_{j=1}^s z^{\gamma_j}.
 \end{aligned}
$$
Clearly, $\mathcal{A}(z)=\mathcal{A}(z^{-1})$
and $\mathcal{B}(z)=\mathcal{B}(z^{-1})$.

Define
\begin{align}
P_1(z)&=\mathcal{A}(z)\mathcal{B}(z)-\mathcal{C}(z^{-1})\mathcal{C}(z),\\
P_{2}(z)&=(\mathcal{A}(z)+2)\mathcal{B}(z)-\mathcal{C}(z^{-1})\mathcal{C}(z),\\
P_{3}(z)&=\mathcal{A}(z)(\mathcal{B}(z)+2)-\mathcal{C}(z^{-1})\mathcal{C}(z), \\
P_{4}(z)&=(\mathcal{A}(z)+2)(\mathcal{B}(z)+2)-\mathcal{C}(z^{-1})\mathcal{C}(z).
\end{align}
Then $P_1(z),P_2(z),P_3(z),P_4(z)$ have same degree and $P_j(z)=P_j(z^{-1})$ for $j=1, 2, 3, 4$.
In the following, we always assume that the degree of $P_j(z)$ is $k$ for $j=1,2,3,4$.

\section{Counting the number of rooted spanning forests}
In this section, we aim to find new formulas for the number of rooted spanning forests
of bicirculant graphs in light of Chebyshev polynomials.

\begin{theorem}\label{th1}
Let $\Gamma=BC(\mathbb{Z}_n;R,T,S)$ be a bicirculant graph. Then we have
\begin{itemize}
\item[(1)] if $\Gamma=\Gamma_1$, then $f_{\Gamma_1}(2n)=|a_k|^n\prod_{\ell=1}^{k}|2T_n(w_{\ell})-2|$,
\item[(2)] if $\Gamma=\Gamma_2$, then $f_{\Gamma_2}(2n)=
|b_k|^\frac{n}{2}|a_k|^\frac{n}{2}\prod_{\ell=1}^{k}|2T_\frac{n}{2}(v_{\ell})+2|\prod_{\ell=1}^{k}|2T_\frac{n}{2}(w_{\ell})-2|$,
\item[(3)] if $\Gamma=\Gamma_3$, then $f_{\Gamma_3}(2n)=
|c_k|^\frac{n}{2}|a_k|^\frac{n}{2}\prod_{\ell=1}^{k}|2T_\frac{n}{2}(u_{\ell})+2|\prod_{\ell=1}^{k}|2T_\frac{n}{2}(w_{\ell})-2|$,
\item[(4)] if $\Gamma=\Gamma_4$, then $f_{\Gamma_4}(2n)=
|d_k|^\frac{n}{2}|a_k|^\frac{n}{2}\prod_{\ell=1}^{k}|2T_\frac{n}{2}(y_{\ell})+2|\prod_{\ell=1}^{k}|2T_\frac{n}{2}(w_{\ell})-2|$,
\end{itemize}
where $a_k,b_k, c_k,d_k$ are the leading coefficients of $P_1(z),P_2(z),P_3(z),P_4(z)$, respectively, and $w_\ell,v_\ell,u_\ell,y_\ell$, $\ell=1,2,\ldots, k$, are all the roots of Chebyshev transform of $P_j(z)=0$ for $j=1,2,3,4$, respectively.
\end{theorem}
\begin{proof}
(1)  By Lemma \ref{adj}, we obtain that the adjacency matrix of the graph $\Gamma_1$ is given by the
$2n \times 2n$ block matrix
$$
A(\Gamma_1)=\begin{pmatrix}
\sum_{j=1}^r\left(Q^{\a_j}+Q^{-\a_j}\right)& \sum_{j=1}^s Q^{\gamma_j} \\
\sum_{j=1}^s Q^{-\gamma_j} & \sum_{j=1}^t\left(Q^{\beta_j}+Q^{-\beta_j}\right)
\end{pmatrix},
$$
where $Q=\operatorname{circ}(0,1,0, \ldots, 0)$. Let $L(\Gamma_1)$ be the Laplacian matrix of $\Gamma_1$. Then
 $$
L(\Gamma_1)+I_{2n}=
\begin{pmatrix}
 k_1I_n-\sum_{j=1}^r\left(Q^{\a_j}+Q^{-\a_j}\right) & -\sum_{j=1}^s Q^{\gamma_j}\\
 -\sum_{j=1}^s Q^{-\gamma_j} & k_2I_n-\sum_{j=1}^t\left(Q^{\beta_j}+Q^{-\beta_j}\right)
 \end{pmatrix},
 $$
 where $k_1=2r+s+1$ and $k_2=2t+s+1$.
 Since the eigenvalues of circulant matrix $Q$ are $1, \V_n, \ldots,\V_n^{n-1}$, where $\V_n=\exp(2\pi \textbf{i}/n)$, then
there exists an invertible matrix $P$ such that
$$
P^{-1}QP=\mathcal{K}_n=
\begin{pmatrix}
1 & 0 & 0 & \cdots & 0 \\
0 & \V_n & 0 & \cdots & 0 \\
0 & 0 & \V_n^2 & \cdots & 0 \\
\vdots & \vdots & \vdots & \ddots & \vdots \\
0 & 0 & 0 & \cdots & \V_n^{n-1}
\end{pmatrix}.
$$
 Then we have
$$
\begin{pmatrix}
P^{-1} & 0 \\
0 & P^{-1}
\end{pmatrix}
(L(\Gamma_1)+I_{2n})
\begin{pmatrix}
P & 0 \\
0 & P
\end{pmatrix}=
\begin{pmatrix}
\mathcal{A}(\mathcal{K}_n) & \mathcal{C}(\mathcal{K}_n) \\
\mathcal{C}(\mathcal{K}_n^{-1}) & \mathcal{B}(\mathcal{K}_n)
\end{pmatrix}.$$
Hence $I_{2n}+L(\Gamma_1)$ and $\mathcal{L}=\begin{pmatrix}
\mathcal{A}(\mathcal{K}_n) & \mathcal{C}(\mathcal{K}_n) \\
\mathcal{C}(\mathcal{K}_n^{-1}) & \mathcal{B}(\mathcal{K}_n)
\end{pmatrix}$ have the same eigenvalues. Suppose $\lambda$ is an eigenvalue of $\mathcal{L}$. Then
$$
0={\rm det}
\begin{pmatrix}
\mathcal{L}-\lambda I_{2n}
\end{pmatrix}
={\rm det}
\begin{pmatrix}
\mathcal{A}(\mathcal{K}_n)-\lambda I_{n} & \mathcal{C}(\mathcal{K}_n) \\
\mathcal{C}(\mathcal{K}_n^{-1}) & \mathcal{B}(\mathcal{K}_n)-\lambda I_{n}
\end{pmatrix}.
$$
Since $(\mathcal{A}(\mathcal{K}_n)-\lambda I_{n})\mathcal{C}(\mathcal{K}_n^{-1})=\mathcal{C}(\mathcal{K}_n^{-1})(\mathcal{A}(\mathcal{K}_n)-\lambda I_{n})$,
by Lemma \ref{mat}, we have
$$
 \begin{aligned}
0&={\rm det}
(
(\mathcal{A}(\mathcal{K}_n)-\lambda I_{n})( \mathcal{B}(\mathcal{K}_n)-\lambda I_{n})-\mathcal{C}(\mathcal{K}_n^{-1})\mathcal{C}(\mathcal{K}_n)
)\\
&={\rm det}
(
\lambda^2 I_n -\lambda I_n(\mathcal{A}(\mathcal{K}_n)+\mathcal{B}(\mathcal{K}_n))+
\mathcal{A}(\mathcal{K}_n)\mathcal{B}(\mathcal{K}_n)-\mathcal{C}(\mathcal{K}_n^{-1})\mathcal{C}(\mathcal{K}_n)
).
 \end{aligned}
$$
Hence, $\lambda$ is a root of the quadratic equation
$$
\lambda^2-\left(\mathcal{A}\left(\varepsilon_n^j\right)+\mathcal{B}(\varepsilon_n^{j})\right)\lambda
+\mathcal{A}\left(\varepsilon_n^j\right) \mathcal{B}\left(\varepsilon_n^{j}\right)-\mathcal{C}\left(\varepsilon_n^{-j}\right) \mathcal{C}\left(\varepsilon_n^{j}\right)=0,
$$
where $j=0,1,2,\ldots,n-1$. The solutions of this equation are $\lambda_{j,1},\lambda_{j,2}$ for $j=0,1,2,\ldots,n-1$.
Then we have
$$
\lambda_{j,1}\lambda_{j,2}=\mathcal{A}\left(\varepsilon_n^j\right) \mathcal{B}\left(\varepsilon_n^{j}\right)-\mathcal{C}\left(\varepsilon_n^{-j}\right) \mathcal{C}\left(\varepsilon_n^{j}\right)
=P_1(\V_n^j).
$$
Therefore,
$$
f_{\Gamma_1}(2n)=\prod_{j=0}^{n-1}\lambda_{j,1}\lambda_{j,2}=\prod_{j=0}^{n-1} P_1\left(\varepsilon_n^j\right),
$$
where $P_1(z)$ is defined in $(3)$. Since $P_1(z)=P_1(z^{-1})$ is a polynomial of degree $k$ with integer
coefficients, by Lemma \ref{ccc}, we have
$$
f_{\Gamma_1}(2n)=(-1)^{nk}a_k^n\prod_{\ell=1}^{k}(2T_n(w_{\ell})-2),
$$\
where $a_k$ is the leading coefficient of $P_1(z)$, and $w_1, w_2,\ldots, w_k$ are all the roots of Chebyshev transform of $P_1(z)=0$.
Since $f_{\Gamma_1}(2n)$ is a positive integer, the result follows.

\medskip

(2) Let $\lambda$ be the eigenvalue of $I_{2n}+L(\Gamma_2)$. Similarly as in the proof of $(1)$,
$\lambda$ is a root of the quadratic equation
$$
\lambda^2-\left(\mathcal{A}\left(\varepsilon_n^j\right)+1-\V_n^{\frac{jn}{2}}+\mathcal{B}\left(\varepsilon_n^{j}\right)\right)\lambda
+\left(\mathcal{A}\left(\varepsilon_n^j\right)+1-\V_n^{\frac{jn}{2}}\right) \mathcal{B}\left(\varepsilon_n^{j}\right)-\mathcal{C}\left(\varepsilon_n^{-j}\right) \mathcal{C}\left(\varepsilon_n^{j}\right)=0,
$$
where $j=0,1,2,\ldots,n-1$. The solutions of this equation are $\lambda_{j,1},\lambda_{j,2}$ for $j=0,1,2,\ldots,n-1$.
Then we have
$$
\lambda_{j,1}\lambda_{j,2}=\left(\mathcal{A}\left(\varepsilon_n^j\right)+1-\V_n^{\frac{jn}{2}}\right) \mathcal{B}\left(\varepsilon_n^{j}\right)-\mathcal{C}\left(\varepsilon_n^{-j}\right) \mathcal{C}\left(\varepsilon_n^{j}\right).
$$
Note that $\lambda_{j,1}\lambda_{j,2}=\mathcal{A}(\varepsilon_n^j) \mathcal{B}(\varepsilon_n^{j})-\mathcal{C}(\varepsilon_n^{-j}) \mathcal{C}(\varepsilon_n^{j})=P_1(\V_n^j)$
for even $j$, and
$\lambda_{j,1}\lambda_{j,2}=\left(\mathcal{A}(\varepsilon_n^j)+2\right) \mathcal{B}(\varepsilon_n^{j})-\mathcal{C}(\varepsilon_n^{-j}) \mathcal{C}(\varepsilon_n^{j})=P_2(\V_n^j)$ for
odd $j$.
Therefore,
$$
\begin{aligned}
f_{\Gamma_2}(2n)&=\prod_{j=0}^{n-1}\lambda_{j,1}\lambda_{j,2}=
\prod_{j=0}^{\frac{n}{2}-1}P_{2}(\V_n^{2j+1})\prod_{j=0}^{\frac{n}{2}-1}P_{1}(\V_n^{2j})\\
&=\frac{\prod_{j=0}^{n-1}P_{2}(\V_n^j)}{\prod_{j=0}^{\frac{n}{2}-1}P_{2}(\V_n^{2j})}\prod_{j=0}^{\frac{n}{2}-1}P_{1}(\V_n^{2j})\\
&=\frac{\prod_{j=0}^{n-1}P_{2}(\V_n^j)}{\prod_{j=0}^{\frac{n}{2}-1}P_{2}(\V_{\frac{n}{2}}^{j})}\prod_{j=0}^{\frac{n}{2}-1}P_{1}(\V_\frac{n}{2}^{j}),\\
\end{aligned}
$$
where $P_1(z)$ and $P_2(z)$ are defined in $(3)$ and $(4)$, respectively. By Lemma \ref{ccc}, we have
$$
\begin{aligned}
\prod_{j=0}^{n-1}P_{2}(\V_n^j)&=(-1)^{nk}b_k^n\prod_{\ell=1}^{k}(2T_n(v_{\ell})-2),\\
\prod_{j=0}^{\frac{n}{2}-1}P_{2}(\V_\frac{n}{2}^j)&=(-1)^{\frac{nk}{2}}b_k^\frac{n}{2}\prod_{\ell=1}^{k}(2T_\frac{n}{2}(v_{\ell})-2),\\
\prod_{j=0}^{\frac{n}{2}-1}P_{1}(\V_\frac{n}{2}^j)&=(-1)^{\frac{nk}{2}}a_k^\frac{n}{2}\prod_{\ell=1}^{k}(2T_\frac{n}{2}(w_\ell)-2),\\
\end{aligned}
$$
where $b_k$ and $a_k$ are the leading coefficients of $P_2(z)$ and $P_1(z)$, respectively, and $v_\ell$ and $w_\ell$, $\ell=1,2,\ldots,k$,
are all the roots of Chebyshev transform of $P_2(z)=0$ and $P_1(z)=0$, respectively.
Since $T_n\left(v_\ell\right)-1=2(T_\frac{n}{2}\left(v_\ell\right)-1)(T_\frac{n}{2}\left(v_\ell\right)+1)$, we have
$$
\begin{aligned}
f_{\Gamma_2}(2n)&=\frac{(-1)^{nk}b_k^n\prod_{\ell=1}^{k}(2T_n(v_{\ell})-2)}
{(-1)^{\frac{nk}{2}}b_k^\frac{n}{2}\prod_{\ell=1}^{k}(2T_\frac{n}{2}(v_\ell)-2)}
(-1)^{\frac{nk}{2}}a_k^{\frac{n}{2}}\prod_{\ell=1}^{k}(2T_{\frac{n}{2}}(w_\ell)-2)\\
&=(-1)^{nk}b_k^\frac{n}{2}a_k^\frac{n}{2}\prod_{\ell=1}^{k}(2T_\frac{n}{2}(v_\ell)+2)\prod_{\ell=1}^{k}(2T_\frac{n}{2}(w_\ell)-2).
\end{aligned}
$$
Since $f_{\Gamma_2}(2n)$ is a positive integer, the result follows.

\medskip

Similarly, we can prove that  both $(3)$ and $(4)$ hold.
\end{proof}

\section{Arithmetic properties of $f_{\Gamma}(2n)$}
In this section, we investigate some arithmetic properties of rooted spanning forests of bicirculant graphs.
Recall that any positive integer $u$ can be uniquely represented in the form $u=vr^2$,
where $u$ and $v$ are positive integers and $v$ is square-free. We will call $v$ the \textbf{square-free part} of $u$.
The main result of this section is the following theorem.

\begin{theorem}\label{th2}
Let $f_{\Gamma}(2n)$ be the number of rooted spanning forests of bicirculant graph $\Gamma$.
Denote by $k_1$ (resp. $k_2$) the number of odd
numbers (resp. even numbers) in $\{\a_1, \a_2, \ldots, \a_r\}$. Denote by $m_1$ (resp. $m_2$) the
number of odd numbers (resp. even numbers) in  $\{\b_1, \b_2, \ldots, \b_t\}$.
Denote by $h_1$ (resp. $h_2$) the
number of odd numbers (resp. even numbers) in  $\{\gamma_1, \gamma_2, \ldots, \gamma_s\}$.
Then we have
\begin{itemize}
\item[(1)] if $\Gamma=\Gamma_1$, then there exist two integer sequences $a_1(n)$ and $b_1(n)$ such that
\begin{eqnarray}
f_{\Gamma_1}(2n)=\begin{cases}
q_1 a_1(n)^2, &{\text{if $n$ is~odd}},\\[1mm]
\ell_1 b_1(n)^2, &{\text{if $n$ is~even}};
\end{cases}
\nonumber
\end{eqnarray}

\item[(2)] if $\Gamma=\Gamma_2$, then there exist two integer sequences $a_2(n)$ and $b_2(n)$ such that
\begin{eqnarray}
f_{\Gamma_2}(2n)=\begin{cases}
q_2 a_2(n)^2, &{\text{if $\frac{n}{2}$ is~odd}},\\[1mm]
\ell_1 b_2(n)^2, &{\text{if $\frac{n}{2}$ is~even}};
\end{cases}
\nonumber
\end{eqnarray}

\item[(3)] if $\Gamma=\Gamma_3$, then there exist two integer sequences $a_3(n)$ and $b_3(n)$ such that
\begin{eqnarray}
f_{\Gamma_3}(2n)=\begin{cases}
q_3 a_3(n)^2, &{\text{if $\frac{n}{2}$ is~odd}},\\[1mm]
\ell_1b_3(n)^2, &{\text{if $\frac{n}{2}$ is~even}};
\end{cases}
\nonumber
\end{eqnarray}

\item[(4)] if $\Gamma=\Gamma_4$, then there exist two integer sequences $a_4(n)$ and $b_4(n)$ such that
\begin{eqnarray}
f_{\Gamma_4}(2n)=\begin{cases}
q_4 a_4(n)^2, &{\text{if $\frac{n}{2}$ is~odd}},\\[1mm]
\ell_1 b_4(n)^2, &{\text{if $\frac{n}{2}$ is~even}},
\end{cases}
\nonumber
\end{eqnarray}
\end{itemize}
where $\ell_1,q_1,q_2,q_3,q_4$ are the square-free parts of $(2s+1)((4k_1+s+1)(4m_1+s+1)-(h_2-h_1)^2), 2s+1,
(2s+1)((4k_1+s+3)(4m_1+s+1)-(h_2-h_1)^2),(2s+1)((4k_1+s+1)(4m_1+s+3)-(h_2-h_1)^2),(2s+1)((4k_1+s+3)(4m_1+s+3)-(h_2-h_1)^2)$, respectively.
\end{theorem}
\begin{proof}
(1) From the proof of Theorem \ref{th1}(1), we have
$f_{\Gamma_1}(2n)=\prod_{j=0}^{n-1} P_1(\varepsilon_n^j).$ Since $\lambda_{j,1}\lambda_{j,2}=P_1(\varepsilon_n^j)=\lambda_{n-j,1}\lambda_{n-j,2}$,
we have
\begin{eqnarray}
f_{\Gamma_1}(2n)=\begin{cases}
P_1\left(1\right)\left(\prod_{j=1}^{\frac{n-1}{2}}\lambda_{j,1}\lambda_{j,2}\right)^2, &{\text{if $n$ is~odd}},\\
P_1\left(1\right)P_1\left(-1\right)\left(\prod_{j=1}^{\frac{n}{2}-1}\lambda_{j,1}\lambda_{j,2}\right)^2, &{\text{if $n$ is~even}}.\\
\end{cases}
\nonumber
\end{eqnarray}
Since $P_1(1)=\mathcal{A}(1)\mathcal{B}(1)-\mathcal{C}(1)\mathcal{C}(1)$ and $P_1(-1)=\mathcal{A}(-1)\mathcal{B}(-1)-\mathcal{C}(-1)\mathcal{C}(-1)$, we have
$P_1(1)=(s+1)^2-(-s)^2=2s+1$ and
$P_1(-1)=\left(2r+s+1-(2k_2-2k_1)\right)\left(2t+s+1-(2m_2-2m_1)\right)-(h_2-h_1)^2
=(4k_1+s+1)(4m_1+s+1)-(h_2-h_1)^2$. Therefore,
\begin{eqnarray}
f_{\Gamma_1}(2n)=\begin{cases}
(2s+1)\left(\prod_{j=1}^{\frac{n-1}{2}}\lambda_{j,1}\lambda_{j,2}\right)^2, &{\text{if $n$ is~odd}},\\
(2s+1)\left((4k_1+s+1)(4m_1+s+1)-(h_2-h_1)^2 \right)\left(\prod_{j=1}^{\frac{n}{2}-1}\lambda_{j,1}\lambda_{j,2}\right)^2, &{\text{if $n$ is~even}}.\\
\end{cases}
\nonumber
\end{eqnarray}
Since each algebraic number $\lambda_{j,i}$ comes into both products $\prod_{j=1}^{\frac{n-1}{2}}\lambda_{j,1}\lambda_{j,2}$ and  $\prod_{j=1}^{\frac{n}{2}-1}\lambda_{j,1}\lambda_{j,2}$ with all of its Galois conjugate elements \cite{DL}, we have $c(n)=\prod_{j=1}^{\frac{n-1}{2}}\lambda_{j,1}\lambda_{j,2}$ and $d(n)=\prod_{j=1}^{\frac{n}{2}-1}\lambda_{j,1}\lambda_{j,2}$ are integers.
Let $q_1$ and $\ell_1$ be the square-free parts of $2s+1$ and $(2s+1)((4k_1+s+1)(4m_1+s+1)-(h_2-h_1)^2)$, respectively.
Then $2s+1=q_1u^2$ and $(2s+1)((4k_1+s+1)(4m_1+s+1)-(h_2-h_1)^2)=\ell_1 v^2$ for some ingeters $u,v$.
Setting $a_1(n)=uc(n)$ for odd $n$ and  $b_1(n)=vd(n)$ for even $n$, we
conclude that number $a_1(n)$ and $b_1(n)$ are integers, and the statement of theorem follows.

\medskip

(2) From the proof of Theorem \ref{th1}(2), we have
$$
\lambda_{j,1}\lambda_{j,2}=\left(\mathcal{A}\left(\varepsilon_n^j\right)+1-\V_n^{\frac{jn}{2}}\right) \mathcal{B}\left(\varepsilon_n^{j}\right)-\mathcal{C}\left(\varepsilon_n^{-j}\right) \mathcal{C}\left(\varepsilon_n^{j}\right)
$$
for $j=0,1,\ldots,n-1$. It follows that $\lambda_{0,1}\lambda_{0,2}=\mathcal{A}(1)\mathcal{B}(1)-\mathcal{C}(1)\mathcal{C}(1)=2s+1$.
Since $\lambda_{j,1}\lambda_{j,2}=\lambda_{n-j,1}\lambda_{n-j,2}$, we have $f_{\Gamma_2}(2n)=(2s+1)\lambda_{\frac{n}{2},1}\lambda_{\frac{n}{2},2}\left(\prod_{j=1}^{\frac{n}{2}-1}\lambda_{j,1}\lambda_{j,2}\right)^2$,
where $\lambda_{\frac{n}{2},1}\lambda_{\frac{n}{2},2}=P_2(-1)$ for odd $\frac{n}{2}$,
and $\lambda_{\frac{n}{2},1}\lambda_{\frac{n}{2},2}=P_1(-1)$ for even $\frac{n}{2}$.
Since $P_1(-1)=(4k_1+s+1)(4m_1+s+1)-(h_2-h_1)^2$ and  $P_2(-1)=(4k_1+s+3)(4m_1+s+1)-(h_2-h_1)^2$,
we have
\begin{eqnarray}
f_{\Gamma_2}(2n)=\begin{cases}
(2s+1)\left((4k_1+s+3)(4m_1+s+1)-(h_2-h_1)^2 \right)\left(\prod_{j=1}^{\frac{n}{2}-1}\lambda_{j,1}\lambda_{j,2}\right)^2, &{\text{if $\frac{n}{2}$ is~odd}},\\
(2s+1)\left((4k_1+s+1)(4m_1+s+1)-(h_2-h_1)^2 \right)\left(\prod_{j=1}^{\frac{n}{2}-1}\lambda_{j,1}\lambda_{j,2}\right)^2, &{\text{if $\frac{n}{2}$ is~even}}.\\
\end{cases}
\nonumber
\end{eqnarray}
Since each algebraic number $\lambda_{j,i}$ comes into the product
 $\prod_{j=1}^{\frac{n}{2}-1}\lambda_{j,1}\lambda_{j,2}$ with all of its Galois conjugate elements \cite{DL},
 we have $d(n)=\prod_{j=1}^{\frac{n}{2}-1}\lambda_{j,1}\lambda_{j,2}$ is an integer.
 Let $q_2$ and $\ell_1$ be the square-free parts of $(2s+1)((4k_1+s+3)(4m_1+s+1)-(h_2-h_1)^2)$ and $(2s+1)((4k_1+s+1)(4m_1+s+1)-(h_2-h_1)^2)$, respectively. Then $(2s+1)((4k_1+s+3)(4m_1+s+1)-(h_2-h_1)^2)=q_2u_1^2$ and $(2s+1)((4k_1+s+1)(4m_1+s+1)-(h_2-h_1)^2)=\ell_1 v_1^2$ for some ingeters $u_1,v_1$. Therefore, $f_{\Gamma_2}(2n)=q_2(u_1d(n))^2$ for odd $\frac{n}{2}$, and $f_{\Gamma_2}(2n)=\ell_1(v_1d(n))^2$ for even $\frac{n}{2}$.
 Setting $a_2(n)=u_1d(n)$ and $b_2(n)=v_1d(n)$, the result follows.

\medskip

Along similar lines, we can prove that both $(3)$ and $(4)$ hold.
\end{proof}

 \section{Asymptotics for $f_{\Gamma}(2n)$}
 In this section, we give asymptotic formulas for the number of rooted spanning forests of
bicirculant graphs. Two functions $f(n)$ and $g(n)$ are said to be \textbf{asymptotically equivalent} if $\text{lim}_{n\rightarrow \infty}\frac{f (n)}{g(n)}=1$. We will write $f (n)\sim g(n)$, $n \rightarrow \infty$ in this case.

\begin{theorem}\label{th3}
Let $\Gamma=BC(\mathbb{Z}_n;R,T,S)$ be a bicirculant graph. Then we have
\begin{itemize}
\item[(1)] if $\Gamma=\Gamma_1$, then
$$f_{\Gamma_1}(2n)\sim A^{n},~~n\rightarrow \infty,$$
where $A=\exp \left(\int_0^1 \log \left|P_1\left(e^{2 \pi \mathbf{i} t}\right)\right| d t\right)$ is the Mahler measure of the Laurent polynomial $P_1(z)$;

\item[(2)] if $\Gamma=\Gamma_2$, then
$$f_{\Gamma_2}(2n)\sim B^{\frac{n}{2}},~~n\rightarrow \infty,$$
where $B=\exp \left(\int_0^1 \log \left|P_2\left(e^{2 \pi \mathbf{i} t}\right)P_1\left(e^{2 \pi \mathbf{i} t}\right)\right| d t\right)$ is the Mahler measure of the Laurent polynomial $P_2(z)P_1(z)$;

\item[(3)] if $\Gamma=\Gamma_3$, then
$$f_{\Gamma_3}(2n)\sim C^{\frac{n}{2}},~~n\rightarrow \infty,$$
where $C=\exp \left(\int_0^1 \log \left|P_3\left(e^{2 \pi \mathbf{i} t}\right)P_1\left(e^{2 \pi \mathbf{i} t}\right)\right| d t\right)$ is the Mahler measure of the Laurent polynomial $P_3(z)P_1(z)$;

\item[(4)] if $\Gamma=\Gamma_4$, then
$$f_{\Gamma_4}(2n)\sim D^{\frac{n}{2}},~~n\rightarrow \infty,$$
where $D=\exp \left(\int_0^1 \log \left|P_4\left(e^{2 \pi \mathbf{i} t}\right)P_1\left(e^{2 \pi \mathbf{i} t}\right)\right| d t\right)$ is the Mahler measure of the Laurent polynomial $P_4(z)P_1(z)$.
\end{itemize}
\end{theorem}
\begin{proof}
We only give the proof for the case $\Gamma=\Gamma_1$, the other cases are similar.
By Theorem \ref{th1}, we obtain
$$
f_{\Gamma_1}(2n)=|a_k|^n\prod_{\ell=1}^{k}|2T_n(w_\ell)-2|,
$$
where $w_\ell=\frac{1}{2}(z_\ell+z_\ell^{-1})$, and $z_\ell,z_\ell^{-1}, \ell=1,2,\ldots,k,$ are all the roots of $P_1(z)=0$.
Let $\theta\in \mathbb{R}$. Then we have
$$P_1(e^{\textbf{i}\theta})=\mathcal{A}(e^{\mathbf{i}\theta}) \mathcal{B}(e^{\mathbf{i}\theta})-\mathcal{C}(e^{\mathbf{i}\theta}) \mathcal{C}(e^{-\mathbf{i}\theta})=\mathcal{A}(e^{\mathbf{i}\theta}) \mathcal{B}(e^{-\mathbf{i}\theta})-|\mathcal{C}(e^{\mathbf{i}\theta})|^2.$$
Since $\mathcal{A}(z)=2r+s+1-\sum_{j=1}^r\left(z^{\a_j}+z^{-\a_j}\right)$ and $
\mathcal{B}(z)=2t+s+1-\sum_{j=1}^t\left(z^{\b_j}+z^{-\b_j}\right)$ and $\mathcal{C}(z)=-\sum_{j=1}^s z^{\gamma_j}$,
we have
$$
\begin{aligned}
\mathcal{A}(e^{\mathbf{i}\theta})
&=2r+s+1-\sum_{j=1}^r \left(e^{\mathbf{i}\theta\a_j}+e^{-\mathbf{i}\theta\a_j}\right)
=s+1+2\sum_{j=1}^r(1-\cos(\theta \a_j))\ge s+1,\\
\mathcal{B}(e^{\mathbf{i}\theta})
&=2t+s+1-\sum_{j=1}^t \left(e^{\mathbf{i}\theta\b_j}+e^{-\mathbf{i}\theta\b_j}\right)
=s+1+2\sum_{j=1}^t(1-\cos(\theta \b_j))\ge s+1,\\
|\mathcal{C}(e^{\mathbf{i}\theta})|^2&=|\sum_{j=1}^s {e^{\mathbf{i}\theta\gamma_j}}|^2=(\sum_{j=1}^s\cos(\theta\gamma_j))^2+(\sum_{j=1}^s\sin(\theta\gamma_j))^2
=s+2\sum_{1\le j<k\le s}\cos((\gamma_j-\gamma_k)\theta) \le s^2.
\end{aligned}
$$
Therefore, we obtain
$$
P_1(e^{\textbf{i}\theta})=\mathcal{A}(e^{\mathbf{i}\theta}) \mathcal{B}(e^{\mathbf{i}\theta})-|\mathcal{C}(e^{\mathbf{i}\theta})|^2
\ge (s+1)^2-s^2=2s+1>0.
$$
Hence $|z_\ell|\not=1$ for $\ell=1,2,\ldots,k$. Replacing $z_\ell$ by $z_\ell^{-1}$ if necessary, we can assume that
$|z_{\ell}|>1$ for all $\ell = 1, 2,\ldots, k$.
Since $T_n(\frac{1}{2}(z_{\ell}+z_{\ell}^{-1}))=\frac{1}{2}(z_{\ell}^n+z_{\ell}^{-n})$, we have
$$T_n(w_\ell)\sim \frac{1}{2}z_\ell^n, \quad |2T_n(w_\ell)-2|\sim |z_{\ell}|^n, ~~n\rightarrow \infty.$$
Then we have
$$
|a_k|^n\prod_{\ell=1}^k\left|2T_n\left(w_\ell\right)-2\right|\sim |a_k|^n\prod_{\ell=1}^k|z_{\ell}|^n
=|a_k|^n\prod_{\substack{P_1(z)=0,\\|z|>1}}|z|^n=A^{n},~~n\rightarrow \infty,
$$
where $A=|a_k|\prod_{P_1(z)=0,|z|>1}|z| $ is the Mahler measure of $P_1(z)$. By (1) and (2), we have
$A=\exp \left(\int_0^1 \log \left|P_1\left(e^{2 \pi \mathbf{i} t}\right)\right| d t\right)$.

Finally,
$$
f_{\Gamma_1}(2n)=|a_k|^n\prod_{\ell=1}^k\left|2 T_n\left(w_\ell\right)-2\right|\sim A^{n}, ~~n\rightarrow \infty.
$$
This completes the proof.
\end{proof}

\section{Examples}
In this section, we give some examples to illustrate our results.

\medskip

\textbf{(1)} \underline{The graph $\Gamma=BC(\mathbb{Z}_n;\{1,-1\},\emptyset,\{0\})$}. In this case, we have $\Gamma=\Gamma_1$. Then $\mathcal{A}(z)=4-(z+z^{-1})$,
$\mathcal{B}(z)=2$ and $\mathcal{C}(z)=-1$. Then $P_1(z)=\mathcal{A}(z)\mathcal{B}(z)-\mathcal{C}(z)\mathcal{C}(z^{-1})=7-2(z+z^{-1})$.
Hence $P_1(z)=0 \Rightarrow \frac{1}{2}(z+z^{-1})=\frac{7}{4}$.

\medskip

(1.1) The number of rooted spanning forests. By Theorem \ref{th1}(1), we have
$$
f_{\Gamma_1}(2n)=2^n\left|2 T_n\left(\frac{7}{4}\right)-2\right|.
$$

\medskip

(1.2) The arithmetic properties of $f_{\Gamma_1}(2n)$.
By Theorem \ref{th2}(1), there exist two integer sequences $a_1(n)$ and $b_1(n)$ such that
\begin{eqnarray}
f_{\Gamma_1}(2n)=\begin{cases}
3a_1(n)^2, &{\text{if $n$ is~odd}},\\
33b_1(n)^2 , &{\text{if $n$ is~even}}.\\
\end{cases}
\nonumber
\end{eqnarray}

\medskip

(1.3) The asymptotics of $f_{\Gamma_1}(2n)$. By Theorem \ref{th3}(1), we have $f_{\Gamma_1}(2n)\sim A^n$, where $A=\frac{1}{2}(7+\sqrt{33})$.

\medskip

\textbf{(2)} \underline{The graph $\Gamma=BC(\mathbb{Z}_n;\{1,-1,\frac{n}{2}\},\emptyset,\{0\})$}. In this case, we have $\Gamma=\Gamma_2$. Then $\mathcal{A}(z)=4-(z+z^{-1})$,
$\mathcal{B}(z)=2$ and $\mathcal{C}(z)=-1$. Then $P_2(z)=(\mathcal{A}(z)+2)\mathcal{B}(z)-\mathcal{C}(z)\mathcal{C}(z^{-1})=11-2(z+z^{-1})$,
$P_1(z)=\mathcal{A}(z)\mathcal{B}(z)-\mathcal{C}(z)\mathcal{C}(z^{-1})=7-2(z+z^{-1})$. Hence $P_2(z)=0$ and $P_1(z)=0$ $ \Rightarrow \frac{1}{2}(z+z^{-1})=\frac{11}{4}$ or $\frac{7}{4}$.

\medskip

(2.1) The number of rooted spanning forests. By Theorem \ref{th1}(2), we have
$$
f_{\Gamma_2}(2n)=2^n\left|2 T_\frac{n}{2}\left(\frac{11}{4}\right)+2\right|\left|2 T_\frac{n}{2}\left(\frac{7}{4}\right)-2\right|.
$$

\medskip

(2.2) The arithmetic properties of $f_{\Gamma_2}(2n)$.
By Theorem \ref{th2}(2), there exist two integer sequences $a_2(n)$ and $b_2(n)$ such that
$$
f_{\Gamma_2}(2n)=\begin{cases}
5a_2(n)^2, &{\text{if $\frac{n}{2}$ is~odd}},\\
33b_2(n)^2 , &{\text{if $\frac{n}{2}$ is~even}}.\\
\end{cases}
$$

\medskip

(2.3) The asymptotics of $f_{\Gamma_2}(2n)$. By Theorem \ref{th3}(2), we have $f_{\Gamma_2}(2n)\sim B^\frac{n}{2}$, where $B=\frac{1}{4}(7+\sqrt{33})(11+\sqrt{105})$.

\medskip

\textbf{(3)} \underline{The graph $\Gamma=BC(\mathbb{Z}_n;\{1,-1\},\{\frac{n}{2}\},\{0\})$}. In this case, we have $\Gamma=\Gamma_3$. Then $\mathcal{A}(z)=4-(z+z^{-1})$,
$\mathcal{B}(z)=2$ and $\mathcal{C}(z)=-1$. Then $P_3(z)=\mathcal{A}(z)(\mathcal{B}(z)+2)-\mathcal{C}(z)\mathcal{C}(z^{-1})=15-4(z+z^{-1})$,
$P_1(z)=\mathcal{A}(z)\mathcal{B}(z)-\mathcal{C}(z)\mathcal{C}(z^{-1})=7-2(z+z^{-1})$. Hence $P_3(z)=0$ and $P_1(z)=0$ $ \Rightarrow \frac{1}{2}(z+z^{-1})=\frac{15}{8}$ or $\frac{7}{4}$.

\medskip

(3.1) The number of rooted spanning forests. By Theorem \ref{th1}(3), we have
$$
f_{\Gamma_3}(2n)=2^\frac{3n}{2}\left|2 T_\frac{n}{2}\left(\frac{15}{8}\right)+2\right|\left|2 T_\frac{n}{2}\left(\frac{7}{4}\right)-2\right|.
$$

\medskip

(3.2) The arithmetic properties of $f_{\Gamma_3}(2n)$.
By Theorem \ref{th2}(3), there exist two integer sequences $a_3(n)$ and $b_3(n)$ such that
$$
f_{\Gamma_3}(2n)=\begin{cases}
69a_3(n)^2, &{\text{if $\frac{n}{2}$ is~odd}},\\
33b_3(n)^2 , &{\text{if $\frac{n}{2}$ is~even}}.\\
\end{cases}
$$

\medskip

(3.3) The asymptotics of $f_{\Gamma_3}(2n)$. By Theorem \ref{th3}(3), we have $f_{\Gamma_3}(2n)\sim C^\frac{n}{2}$, where $C=\frac{1}{4}(7+\sqrt{33})(15+\sqrt{161})$.

\medskip

\textbf{(4)} \underline{The graph  $\Gamma=BC(\mathbb{Z}_n;\{1,-1,\frac{n}{2}\},\{\frac{n}{2}\},\{0\})$}. In this case, we have $\Gamma=\Gamma_4$. Then $\mathcal{A}(z)=4-(z+z^{-1})$,
$\mathcal{B}(z)=2$ and $\mathcal{C}(z)=-1$. Then $P_4(z)=(\mathcal{A}(z)+2)(\mathcal{B}(z)+2)-\mathcal{C}(z)\mathcal{C}(z^{-1})=23-4(z+z^{-1})$,
$P_1(z)=\mathcal{A}(z)\mathcal{B}(z)-\mathcal{C}(z)\mathcal{C}(z^{-1})=7-2(z+z^{-1})$. Hence $P_4(z)=0$ and $P_1(z)=0$ $ \Rightarrow \frac{1}{2}(z+z^{-1})=\frac{23}{8}$ or $\frac{7}{4}$.

\medskip

(4.1) The number of rooted spanning forests. By Theorem \ref{th1}(4), we have
$$
f_{\Gamma_4}(2n)=2^\frac{3n}{2}\left|2 T_\frac{n}{2}\left(\frac{23}{8}\right)+2\right|\left|2 T_\frac{n}{2}\left(\frac{7}{4}\right)-2\right|.
$$

\medskip

(4.2) The arithmetic properties of $f_{\Gamma_4}(2n)$.
By Theorem \ref{th2}(4), there exist two integer sequences $a_4(n)$ and $b_4(n)$ such that
$$
f_{\Gamma_4}(2n)=\begin{cases}
93a_4(n)^2, &{\text{if $\frac{n}{2}$ is~odd}},\\
33b_4(n)^2 , &{\text{if $\frac{n}{2}$ is~even}}.\\
\end{cases}
$$

\medskip

(4.3) The asymptotics of $f_{\Gamma_4}(2n)$. By Theorem \ref{th3}(4), we have $f_{\Gamma_4}(2n)\sim D^\frac{n}{2}$, where $D=\frac{1}{4}(7+\sqrt{33})(23+\sqrt{465})$.

\medskip
At last, we give two examples
on Cayley graphs over dihedral groups.
Let $\mathrm{D}_{2n}=\l a,b \mid a^n=b^2=e, bab=a^{-1}\r$ be the dihedral group.
Note that Cayley graphs over dihedral groups can be seen as bi-Cayley graphs over cyclic groups.
Our results are also effective to deal with such kind of graphs.

\medskip

\textbf{(5)} \underline{The graph  $\Gamma=BC\{\mathbb{Z}_n;\{1,-1\},\{1,-1\},\{0\}\}$}. In this case, $\Gamma=\Gamma_1\cong \Cay(\D_{2n},\{a,a^{-1},b\})$.  Then we have
$\mathcal{A}(z)=\mathcal{B}(z)=4-(z+z^{-1})$ and $\mathcal{C}(z)=-1$. Hence $P_1(z)=\mathcal{A}(z) \mathcal{B}\left(z\right)-\mathcal{C}(z^{-1}) \mathcal{C}\left(z\right)=(4-(z+z^{-1}))^2-1$. Then $P_1(z)=0 \Rightarrow \frac{1}{2}(z+z^{-1})=\frac{3}{2}$ or $\frac{5}{2}$.

\medskip

(5.1) The number of rooted spanning forests. By Theorem \ref{th1}(1), we have
$$
f_{\Gamma_1}(2n)=\left|2 T_n\left(\frac{3}{2}\right)-2\right| \left|2 T_n\left(\frac{5}{2}\right)-2\right|.
$$

\medskip

(5.2) The arithmetic properties of $f_{\Gamma_1}(2n)$.
By Theorem \ref{th2}(1), there exist two integer sequences $a_1(n)$ and $b_1(n)$ such that
$$
f_{\Gamma_1}(2n)=\begin{cases}
3a_1(n)^2, &{\text{if $n$ is~odd}},\\[1mm]
105b_1(n)^2 , &{\text{if $n$ is~even}}.
\end{cases}
$$

\medskip

(5.3) The asymptotics of $f_{\Gamma_1}(2n)$. By Theorem \ref{th3}(1), we have $f_{\Gamma_1}(2n)\sim A^n$, where $A=\frac{1}{4}(3+\sqrt{5})(5+\sqrt{21})$.

\medskip

\textbf{(6)} \underline{The graph  $\Gamma=BC\{\mathbb{Z}_n;\{1,-1,\frac{n}{2}\},\{1,-1,\frac{n}{2}\},\{0\}\}$}. In this case, we have $\Gamma=\Gamma_4\cong\Cay(\D_{2n},\{a,a^{-1}, a^{\frac{n}{2}}, b\})$. Then
$\mathcal{A}(z)=\mathcal{B}(z)=4-(z+z^{-1})$ and $\mathcal{C}(z)=-1$. It follows that $P_1(z)=(4-(z+z^{-1}))^2-1$ and
$P_4(z)=(6-(z+z^{-1}))^2-1$. Hence $P_1(z)=0$  and $P_4(z)=0 \Rightarrow $
$\frac{1}{2}(z+z^{-1})=\frac{3}{2}, \frac{5}{2}$ or $\frac{1}{2}(z+z^{-1})=\frac{5}{2}, \frac{7}{2}$.

\medskip

(6.1) The number of rooted spanning forests. By Theorem \ref{th1}(4), we have
$$
f_{\Gamma_4}(2n)=\left|2T_{\frac{n}{2}}\left(\frac{5}{2}\right)+2\right|\left|2T_{\frac{n}{2}}\left(\frac{7}{2}\right)+2\right|
\left|2T_{\frac{n}{2}}\left(\frac{3}{2}\right)-2\right|\left|2T_{\frac{n}{2}}\left(\frac{5}{2}\right)-
2\right|.
$$

\medskip

(6.2) The arithmetic properties of $f_{\Gamma_4}(2n)$.
By Theorem \ref{th2}(4), there exist two integer sequences $a_4(n)$ and $b_4(n)$ such that
$$
f_{\Gamma_4}(2n)=\begin{cases}
21a_4(n)^2, &{\text{if $\frac{n}{2}$ is~odd}},\\[1mm]
105b_4(n)^2 , &{\text{if $\frac{n}{2}$ is~even}}.
\end{cases}
$$

\medskip

(6.3) The asymptotics of $f_{\Gamma_4}(2n)$. By Theorem \ref{th3}(4), we have $f_{\Gamma_4}(2n)\sim D^\frac{n}{2}$, where
$D=\frac{1}{16}(3+\sqrt{5})(5+\sqrt{21})^2(7+3\sqrt{5})$.

\section*{Declaration of competing interests}
We declare that we have no conflict of interests to this work.

\section*{Data availability}
No data was used for the research described in the article.


 \end{document}